\documentclass[12pt,a4paper,reqno,oneside]{amsart}
\usepackage{t1enc}
\usepackage{times}
\usepackage{amssymb}
\usepackage[mathscr]{euscript}
\usepackage{graphicx}
\usepackage[pdftex]{color} 

\newtheorem{thm}{Theorem}[section]
\newtheorem{prop}[thm]{Proposition}
\newtheorem{lem}[thm]{Lemma}
\newtheorem{cor}[thm]{Corollary}

\theoremstyle{remark}
\newtheorem{rem}[thm]{Remark}

\theoremstyle{definition}
\newtheorem{defn}[thm]{Definition}

\renewcommand{\phi}{\varphi}
\newcommand{\tr}{\mathrm{Tr}}
\newcommand{\E}{\mathrm{E}}
\newcommand{\<}{\langle}
\renewcommand{\>}{\rangle}

\newcommand{\praum}{(\Omega,\mathfrak A,P)}
\newcommand{\Log}{\mathrm{Log}}
\newcommand{\Cov}{\mathrm{Cov}}

\begin{document}

\title{Subordination of Hilbert space valued L\'evy processes}
\date{\today}

\author[Benth]{Fred Espen Benth}
\address[Fred Espen Benth]{\\
Centre of Mathematics for Applications \\
University of Oslo\\
P.O. Box 1053, Blindern\\
N--0316 Oslo, Norway}
\email[]{fredb\@@math.uio.no}
\urladdr{http://folk.uio.no/fredb/}
\author[Kr\"uhner]{Paul Kr\"uhner}
\address[Paul Kr\"uhner]{\\
Department of Mathematics \\
University of Oslo\\
P.O. Box 1053, Blindern\\
N--0316 Oslo, Norway}
\email[]{paulkru\@@math.uio.no}


\thanks{This paper has been developed under financial support of the project "Managing Weather Risk in
Electricity Markets" (MAWREM), funded by the RENERGI-program of the Norwegian Research Council.} 
\thanks{The authors want to express their thanks to Prof. Ole Barndorff-Nielsen for fruitful comments.}

\begin{abstract}
We generalise multivariate subordination of L\'evy processes as introduced by Barndorff-Nielsen, Pedersen, and 
Sato~\cite{barndorff.et.al.01} to Hilbert space valued L\'evy processes. 
The processes are explicitly characterised and conditions for integrability and martingale properties are derived under various 
assumptions of the L\'evy process and subordinator.  As an application of our theory we construct explicitly some    
Hilbert space valued versions of L\'evy processes which are popular in the univariate and multivariate case. In particular,  
we define a normal inverse Gaussian L\'evy process in Hilbert space as a subordination of a Hilbert space valued Wiener process
by an inverse Gaussian L\'evy process. The resulting process has the property that at each time all its finite dimensional projections are multivariate normal inverse Gaussian distributed as introduced in Rydberg~\cite{rydberg.97}. 
\end{abstract}

\maketitle


\section{introduction}
Subordination, which was first introduced by Bochner~\cite{bochner.49}, has become a widely used tool to construct new Markov processes or $C_0$-semigroups. Barndorff-Nielsen, Pedersen and Sato~\cite{barndorff.et.al.01} extended this approach to multivariate subordination of L\'evy processes, i.e.\ subordination of $d$ independent L\'evy processes $L_1,\dots,L_d$ with $d$ possibly dependent subordinators $\Theta_1,\dots,\Theta_d$. They proved that the resulting process $X(t):=(L_1(\Theta_1(t)),\dots,L_d(\Theta_d(t)))$ is again a L\'evy process and its characteristics as well as its L\'evy exponent can be expressed easily in terms of properties of $L$ and $\Theta$. In the recent paper of Mendoza-Arriaga and Linetsky~\cite{Mendoza.Linetsky.12} multivariate subordination has been generalised to Markov processes with locally compact state spaces. Baeumer, Kov\'acs and Meerschaert~\cite{baeumer.et.al.08} treated multivariate subordination from an analytical point of view. 

Peszat and Zabzcyk~\cite[page 62]{peszat.zabczyk.07} indicate that the usual subordination procedure can be used to generate new
Hilbert space valued L\'evy processes. We follow their suggestion, and introduce multivariate subordination of Hilbert space 
valued L\'evy processes. In particular, we subordinate a cylindrical Brownian motion with an inverse Gaussian process which generalises subordination of real valued Brownian motions with the same subordinator. The latter subordinated process is a so-called normal inverse Gaussian L\'evy process, while the first an infinite dimensional generalization of it. As it turns out, projections of this 
Hilbert space valued normal inverse Gaussian L\'evy process to finite dimensional subspaces become multivariate
normal inverse Gaussian distributed L\'evy processes (cf. \O igard and Hanssen~\cite{oigard.hanssen.02} for 
definition of the multivariate normal inverse Gaussian distribution). We also introduce $\alpha$-stable and 
variance Gamma processes in infinite dimensions. Hilbert space valued L\'evy processes can be applied
to modeling of the spatio-temporal dynamics of weather variables like wind and temperature, and the evolution of futures prices in
energy markets or forward rates in fixed-income markets. Other areas of application includes quantum physics and turbulence.

The subordinated L\'evy processes can be completely characterised by the characteristics of the L\'evy process and subordinator. Moreover,
we analyse in detail the integrability properties of the Hilbert space valued subordinated L\'evy process. Finite first and second moments
of these infinite dimensional processes can be shown to exists under various mild conditions on the L\'evy process and/or the 
subordinator process. We derive several different conditions under which (square-)integrability holds.

This paper is arranged as follows. In the second section we introduce multivariate subordination of Hilbert space valued L\'evy processes and give formulas for the characteristic function and the characteristics of the subordinated process. In the third section we characterise the second order moment structure and characterise martingale property of the subordinated process. In the fourth section Hilbert space valued normal inverse Gaussian processes (and other Hilbert space valued L\'evy processes) are introduced and we apply the results of the previous sections to them.

\subsection{Mathematical preliminaries}
$\mathbb R$, resp.\ $\mathbb C$, denotes the real, resp.\ the complex number, and $\mathbb R_+:=[0,\infty)$ (resp.\ $\mathbb R_-:=(-\infty,0]$) the non-negative (resp.\ non-positive) real numbers. $\praum$ will always denote a probability space. If not otherwise stated, we will allways assume that our stochastic processes have c{\`a}dl{\`a}g paths and work with the truncation function $\chi(x):=x1_{\{\vert x\vert\leq 1\}}$. 

Throughout this article let $d\in\mathbb N$, $(H_j,\<\cdot\vert\cdot\>_j)$ be separable Hilbert space and $L_j$ be an $H_j$-valued L\'evy process for $j=1,\dots,d$ such that $L_1,\dots,L_d$ are independent, cf.\ Peszat and Zabczyk~\cite[Section 4]{peszat.zabczyk.07}. Let $(b_j,Q_j,\nu_j)$ be the characteristics of $L_j$ (we provide a proof of the uniqueness of the characteristics in Lemma \ref{L:Eindeutigkeit des Tripel}) and denote the L\'evy exponent of $L_j$ by $\phi_j$ for all $j=1,\dots,d$, i.e.\ $\phi_j:H_j\rightarrow\mathbb C$ such that
$$\E e^{i\<L_j(t)\vert u\>} = \exp(t\phi_j(u))$$
for any $t\in\mathbb R_+$, $u\in H_j$ (cf.\ Peszat and Zabcyzk~\cite[Section 4.6]{peszat.zabczyk.07}). Define $L:=(L_1,\dots,L_d)$, $H:=H_1\otimes\dots\otimes H_d$ and $\<u\vert v\>:=\sum_{j=1}^d\<u_j\vert v_j\>_j$ for $u,v\in H$. Let $(b,Q,\nu)$ be the characteristics of $L$. For $\theta\in\mathbb R_+^d$ we define $L(\theta):=(L_1(\theta_1),\dots,L_d(\theta_d))$. For $a\in\mathbb R^d$ and $u\in H$ we define $au:=(a_1u_1,\dots,a_du_d)\in H$. For bounded linear operators $T_1,\dots,T_d$ on $H_1,\dots,H_d$ we define $T_1\times\cdots\times T_d:H\rightarrow H,u\mapsto (T_1u_1,\dots,T_du_d)$ and for $a\in\mathbb R_+^d$ and $T:=T_1\times \cdots\times T_d$ we also define $aT:=a_1T_1\times\cdots\times a_dT_d$.

Let $\Theta$ be a L\'evy process with values in $\mathbb R^d$ such that $\Theta_j$ is a subordinator for all $j=1,\dots,d$, cf.\ Sato~\cite[Definition 21.4]{sato.99} or Skorokhod~\cite{skorohod.91}. Let $(a,c,F)$ be the L\'evy-Khintchine triplet of $\Theta$. Then $c=0$ and $\int_{\{\vert \theta\vert\leq1\}}\theta F(d\theta)<\infty$ since the paths of $\Theta$ are of bounded variation, cf.\ \cite[Theorem 21.9]{sato.99}.  Define $a_0:=a-\int_{\mathbb R_+^d} \chi(\theta)F(d\theta)$ and
$$\psi:(\mathbb R_-+i\mathbb R)^d\rightarrow\mathbb C,s\mapsto a_0s+\int_{\mathbb R_+^d} (e^{s\theta}-1)F(d\theta).$$
From \cite[Theorem 8.1]{sato.99} it can be seen that $Ee^{s\Theta(1)} = \exp(\psi(s))$ for any $s\in(\mathbb R_-+i\mathbb R)^d$ and \cite[Theorem 21.5]{sato.99} yields $a_0\in(\mathbb R_+)^d$ and $F$ is concentrated on $(\mathbb R_+)^d$.

Further unexplained notation is used as in the books of Jacod and Shiryaev~\cite{js.87} and Peszat and Zabczyk~\cite{peszat.zabczyk.07}.

\begin{rem}
Like in the finite dimensional case there is a connection between the characteristics of a L\'evy process and its L\'evy exponent. Indeed, \cite[Theorem 4.27]{peszat.zabczyk.07} yields
$$\phi_j(u) = i\<u\vert b_j\>_j -\frac{1}{2}\<Q_ju\vert u\>_j +\int_{H_j}\left(e^{i\<u\vert x\>_j}-1-i\<u\vert\chi(x)\>_j\right)\nu_j(dx)$$
for any $u\in H_j$ and any $j=1,\dots,d$. Moreover, the triplet of $L=(L_1,\dots,L_d)$ can of course be expressed in the triplets of $L_1,\dots,L_d$. Namely we have
\begin{eqnarray*}
 b &=& (b_1,\dots,b_d) \\
 Q &=& Q_1\times\dots\times Q_d \\
 \nu(A) &=& \sum_{j=1}^d\nu_j^{\eta_j}(A)
\end{eqnarray*}
for any $A\subseteq \mathcal B(H)$ where $\eta_j$ is the natural embedding from $H_j$ into $H$, e.g.\ $\eta_1:H_1\rightarrow H,u\mapsto (u,0,\dots,0)$.
\end{rem}

As a sideremark we want to note that $H$ is a modul over the ring $(\mathbb R^d,+,\cdot)$ with respect to the multiplication $(a,u)\mapsto au$ as defined above where $\cdot$ is the componentwise multiplication on $\mathbb R^d$. The mapping $Q$ is an $\mathbb R^d$-linear mapping.

The case $d=1$ will be of special interest in this article and after Section 3 the results for this particular case will be used only.

\section{Subordinated Hilbert space valued L\'evy processes}\label{A:Theorie Subordination}
Multivariate subordination of $\mathbb R^d$-valued L\'evy processes has been treated in Barndorff-Nielsen, Pedersen and Sato~\cite{barndorff.et.al.01}. We extend their results to Hilbert space valued processes, and
define the {\em multivariate subordinated L\'evy process}
\begin{eqnarray}\label{e:Subordinierter Prozess}
X(t):= (L_1(\Theta_1(t)),\dots,L_d(\Theta_d(t))) = L(\Theta(t))
\end{eqnarray}
for any $t\geq0$.

As we shall see in this Section, the L\'evy exponent of the subordinated L\'evy process can be easily expressed in the L\'evy exponent of the original L\'evy processes and the Laplace exponent of the subordinator, see Theorem \ref{S:Fourier-Darstellung} below. Moreover, the characteristics of the subordinated L\'evy process can be expressed in terms of the characteristics of the original L\'evy processes, the characteristics of the subordinators and the distribution of the original L\'evy processes, see Theorem \ref{S:subordiniertes tripel} below.

\begin{rem}
 The process $X$ has c\`adl\`ag paths because $\Theta_1,\dots,\Theta_d$ have c\`adl\`ag paths and $L_1,\dots,L_d$ have c\`adl\`ag paths.
\end{rem}

\begin{rem}\label{R:characteristische Funktion}
Observe that the set of functions 
$$\{f_u:H\rightarrow\mathbb C,x\mapsto e^{i\<u\vert x\>}:u\in H\}$$
is a monotone class. Hence \cite[Corollary A.4.4]{ethier.kurtz.86} yields that the law of $H$-valued random variables $Y,Z$ coincide if and only if
 $$\E(e^{i\<u\vert Y\>}) =\E(e^{i\<u\vert Z\>})$$ for any $u\in H$.
\end{rem}

\begin{thm}\label{S:Fourier-Darstellung}

 The process $X$ is a L\'evy process and its L\'evy exponent is given by
$$\rho:H\rightarrow\mathbb C,u\mapsto \psi\left((\phi_1(u_1),\dots,\phi_d(u_d)\right).$$
\end{thm}
\begin{proof}
This proof is along the lines of the proof of \cite[Theorem 3.3]{barndorff.et.al.01}. Let $n\in\mathbb N$, $u\in H^n$ and define
$$T_n:=\{\theta\in\mathbb R_+^{(n+1)\times d}: \theta_{k,j}<\theta_{k+1,j} \text{ for any }k=1,\dots,n, j=1,\dots,d\}.$$
Let 
$$f:T_n\rightarrow\mathbb C,\theta\mapsto \E\exp\left(i\sum_{k=1}^n\<u_{k}\vert (L(\theta_{k+1})-L(\theta_{k})\>\right).$$
Independence of the coordinates of $L$, independence of the increments of $L$ and \cite[Theorem 4.27]{peszat.zabczyk.07} yield
\begin{eqnarray*}
f(\theta) &=& \prod_{k=1}^n\prod_{j=1}^d\exp((\theta_{k+1,j}-\theta_{k,j})\phi_j(u_{k,j})) \\
  &=& \exp\left(\sum_{k=1}^n\sum_{j=1}^d(\theta_{k+1,j}-\theta_{k,j})\phi_j(u_{k,j})\right)
\end{eqnarray*}
for any $\theta\in T_n$. Since $L$ and $\Theta$ are independent we get
\begin{eqnarray*}
 &&\E\exp\left(i\sum_{k=1}^n\<u_k\vert L(\Theta(t_{k+1}))-L(\Theta(t_{k}))\>\right) \\
 &=& \E f((\Theta_j(t_k))_{k\in\{1,\dots,n+1\},j\in\{1,\dots,d\}}) \\
 &=& \E\left(\exp\left(\sum_{k=1}^n\sum_{j=1}^d(\Theta_{j}(t_{k+1})-\Theta_{j}(t_k))\phi_j(u_{k,j})\right)\right) \\
 &=& \prod_{k=1}^n\E\left(\exp\left(\sum_{j=1}^d(\Theta_{j}(t_{k+1})-\Theta_{j}(t_k))\phi_j(u_{k,j})\right)\right) \\
 &=& \prod_{k=1}^n\exp\left((t_{k+1}-t_k)\psi((\phi_{j}(u_{k,j}))_{j=1,\dots,d})\right) \\
 &=& \prod_{k=1}^n\exp\left((t_{k+1}-t_k)\rho(u_k)\right)
\end{eqnarray*}
for any $0\leq t_1<\dots<t_{n+1}$. Now it follows that $X$ is a L\'evy process.

Moreover, for $n=1$, $t_2=1,t_1=0$ we have
$$\E\exp\left(i\<u\vert L(\Theta(1))\>\right) = \exp\left(\rho(u))\right)$$
which is the claimed formula.
\end{proof}

We are now ready to compute the characteristics of the multivariate subordinated L\'evy process $X$.
\begin{thm}\label{S:subordiniertes tripel}
 We have $\int_{\mathbb R_+^d} \vert\E(\chi(L(\theta)))\vert F(d\theta)<\infty$. Define
\begin{eqnarray*}
 \beta &=& a_{0}b+\int_{\mathbb R_+^d}\E(\chi(L(\theta))) F(d\theta),\\
 \Gamma &=& a_0Q \quad\text{and}\\
 \mu(A) &=& \sum_{j=1}^da_{0,j}\nu_j^{\eta_j}(A) + \int_{\mathbb R_+^d} P^{L(t)}(A) F(d\theta) 
\end{eqnarray*}
for any Borel-sets $A\subseteq H$. Then $(\beta,\Gamma,\mu)$ is the characteristics of $X$.
\end{thm}
\begin{proof}
Define the measure 
$$\widetilde\mu(A):=\int_{\mathbb R_+^d} P^{L(\theta)}(A) F(d\theta)$$ 
for any Borel-sets $A\subseteq H$. Observe that for any measurable function $f:H\rightarrow\mathbb R$ which is positive or $\widetilde\mu$-integrable we have
$$\int_H fd\widetilde\mu = \int_{\mathbb R_+^d}\E(f(L(\theta)))F(d\theta).$$
By Lemma \ref{L:Wachstum chi} there is $C>1$ such that $\vert\E\chi(L(\theta))\vert \leq \vert \theta\vert C$ for any $\theta\in\mathbb R_+^d$. Thus $\theta\mapsto \vert\E\chi(L(\theta))\vert$ is bounded by $(1\wedge\vert \theta\vert)C$. Hence \cite[Theorem 21.5]{sato.99} yields that it is $F$-integrable which is the first part of the claim. Theorem \ref{S:Fourier-Darstellung} yields that the L\'evy exponent of $X$ is given by
$$\rho(u) := \psi((\phi_j(u_j))_{j=1,\dots,d})$$
for any $u\in H$. Then
\begin{eqnarray*}
 \rho(u) &=& \psi((\phi_j(u_j))_{j=1,\dots,d}) \\
 &=& \sum_{j=1}^d a_{0,j}\phi_j(u_j) + \int_{\mathbb R_+^d} \left(e^{\sum_{j=1}^d\theta_j\phi_j(u_j)} -1\right) F(d\theta)
\end{eqnarray*}
for any $u\in H$. Moreover,
\begin{eqnarray*}
 && \int_{\mathbb R_+^d} \left(e^{\sum_{j=1}^d\theta_j\phi_j(u_j)} -1\right) F(d\theta) \\
 &=& \int_{\mathbb R_+^d} \left(\E\left(\exp\left(i\<L(\theta)\vert u\>\right)\right) -1-i\<\E\chi(L(\theta))\vert u\>\right) F(d\theta)+i\<\gamma\vert u\> \\
 &=& \int_{\mathbb R_+^d} \E\left(\exp\left(i\<L(\theta)\vert u\>\right)-1-i\<\chi(L(\theta))\vert u\>\right) F(d\theta)+i\<\gamma\vert u\>
\end{eqnarray*}
where $\gamma = \int_{\mathbb R_+^d}\E\chi(L(\theta)) F(d\theta)$ for any $u\in H$. Let $u\in H$ and define
$$f:H\rightarrow\mathbb R_+,x\mapsto \exp\left(i\<u\vert x\>\right)-1-i\<\chi(x)\vert u\>.$$
Lemma \ref{L:Wachstum g} yields $g(\theta):=\E\vert f(L(\theta))\vert\leq \vert\theta\vert C_2$ for some $C_2>0$ and any $\theta\in\mathbb R_+^d$. Since $g$ is positive and bounded by some constant $C_3$, we have
\begin{eqnarray*}
 \int_H \vert f(x)\vert\widetilde\mu(dx) &=& \int_{\mathbb R_+^d} g(\theta) F(d\theta) \\
 &\leq& \int_{\mathbb R_+^d}(1\wedge \vert\theta\vert) F(d\theta) (C_2\vee C_3) \\
 &<&\infty
\end{eqnarray*}
Thus $f$ is $\widetilde\mu$-integrable. Hence we have
$$\int_{\mathbb R_+^d} \left(e^{\sum_{j=1}^d\theta_j\phi_j(u_j)} -1\right) F(d\theta) = \int_Hf(x)\widetilde\mu(dx) + i\<\gamma\vert u\>.$$
\cite[Theorem 4.27]{peszat.zabczyk.07} implies that the characteristics can be read from the representation
$$\rho(u) = \sum_{j=1}^d a_{0,j}\phi_j(u_j) + \int_Hf(x)\widetilde\mu(dx) + i\<\gamma\vert u\>.$$
Hence, the proof is complete.
\end{proof}

\section{Probabilistic features of subordinated L\'evy processes}\label{A:Momente und Martingale}
In this section we want to investigate the probabilistic features of the subordinated L\'evy process $X(t)=L(\Theta(t))$, i.e.\ we give necessary and sufficient conditions for $X$ to have finite first or second moment and provide formulas for those moments. Thanks to \cite[Section 4.9]{peszat.zabczyk.07} one can characterise finiteness of the second moment of $X$ completely in terms of moments of $L$ and $\Theta$. It turns out that square integrability of $L$ and $\Theta$ are sufficient and essentially necessary for square integrability of $X$ if $L$ is not a martingale. If $L$ is a square-integrable martingale, then integrability of $\Theta$ is sufficient and essentially necessary to ensure that $X$ is square integrable, cf.\ Theorem \ref{S:charakterisierung quadint} below. We also show that $X$ is integrable if $L$ and $\Theta$ are integrable where we make use of several results collected in the Appendix, cf.\ Theorem \ref{S:charakterisierung von Integrierbarkeit} below. If $L$ is a square-integrable martingale, then it is sufficient that $\sqrt{\vert\Theta(1)\vert}$ is integrable, cf.\ Theorem \ref{S:X Martingal} below. However, the authors do not know if, under the assumption that $L$ is square-integrable, the integrability of $\sqrt{\vert\Theta(1)\vert}$ is necessary for integrability of $X$. Corollary \ref{C:int im Guasschen fall} shows that this is true if $L$ is a cylindrical Brownian motion. If $L$ is a martingale but not square-integrable, then it is possible that integrability of $\sqrt{\vert\Theta(1)\vert}$ is not sufficient to ensure integrability of $X$ as we will show in Proposition \ref{P:asLP Wachstumsfunktion} at the end of Section \ref{A:alpha stabil}.

\begin{rem}
 Let $(\mathcal F_t)_{t\in\mathbb R_+}$ be a filtration such that $L(t)-L(s)$ is independent of $\mathcal F_s$. Then the following statements are equivalent.
\begin{itemize}
\item $L$ is a $(\mathcal F_t)_{t\in\mathbb R_+}$-martingale.
\item $L$ is a martingale w.r.t.\ its own (right-continuous) filtration.
\item $L$ is mean zero, i.e.\ $L$ has finite expectation and $\E L(t)=0$ for any $t\in\mathbb R_+$.
\item $\int_{\{\vert x\vert>1\}}\vert x\vert \nu(dx)<\infty$ and $0=b+\int_{\{\vert x\vert>1\}} x\nu(dx)$.
\end{itemize}
\end{rem}

\begin{defn}
 Let $Y$ be any $H$-valued random variable with finite second moment. Then the {\em covariance operator} $\Cov(Y)$ of $Y$ is defined by the equation
 $$\<\Cov(Y)x\vert y\> = \E(\<Y-\E Y\vert x\>\<Y-\E Y\vert y\>)$$
for any $x,y\in H$.
\end{defn}

We first recall some properties of square integrable L\'evy processes.
\begin{prop}\label{P:Quadratstruktur}
 The L\'evy process $L$ is square integrable if and only if $\int_{H}\vert x\vert^2 \nu(dx)<\infty$. If $L$ is square integrable, then 
\begin{itemize}
 \item $\E(L(t)) = t\left(b+\int_{\vert x\vert>1}x\nu(dx)\right)$,
 \item $\E(\vert L(t)-\E L(t)\vert^2) = t\left(\tr(Q)+\int_H\vert x\vert^2\nu(dx)\right)$,
 \item $M(t):= L(t)-t\E L(1)$ is a mean zero and square integrable L\'evy process with the characteristics $(b-\E L(1),Q,\nu)$ and
 \item $\<\Cov(L(t)) x\vert y\> = \<Q x\vert y\> + \int_H \<x\vert z\>\<y\vert z\> \nu(dz)$ for any $x,y\in H$.
\end{itemize}
\end{prop}
\begin{proof}
See \cite[Theorem 4.47 and Theorem 4.49]{peszat.zabczyk.07}.
\end{proof}
\begin{rem}
The covariance operator of $L$ can, of course, be expressed in the covariance operators of $L_1,\dots,L_d$. We have
 $$\Cov(L(1)) = \Cov(L_1(1))\times\cdots\times \Cov(L_d(1)).$$
\end{rem}

We first aim at characterising square integrability of the subordinated process $X$, see Theorem \ref{S:charakterisierung quadint} below. The proof is devided into three parts where the next two lemmas each contain a part. It is essentially necessary that $L$ is square integrable for $X$ being square integrable. However, square integrability of the multivariate subordinator $\Theta$ is only needed if $L$ is not a martingale. If $L$ is a square integrable martingale, then integrability for $\Theta$ is sufficient to ensure that $X$ is square integrable. This is the statement of the next Lemma.

\begin{lem}\label{L:Quadrat Martingal}
 Let $L$ be mean zero and square integrable and $\Theta$ be integrable. Then $X$ is mean zero and square integrable and 
$$\Cov(X(1)) = \E(\Theta(1))\Cov(L(1)).$$
\end{lem}
\begin{proof}
 Let $g(\theta):=\E(\vert L(\theta)\vert^2)$ for any $\theta\in\mathbb R_+^d$. Proposition \ref{P:Quadratstruktur} yields
$$g(\theta)=\sum_{j=1}^d\theta_j\E(\vert L_j(1)\vert^2)$$
for any $\theta\in\mathbb R_+^d$. By conditioning on $\Theta$ we get
\begin{eqnarray*}
\E(\vert X(t)\vert^2) &=& \E g(\Theta(t)) \\
 &=& \sum_{j=1}^d\E \Theta_j(t) \E(\vert L_j(1)\vert^2) \\
 &<& \infty
\end{eqnarray*}
for any $t\geq0$. Thus $X$ is square integrable. Conditioning on $\Theta$ yields $\E X(t)=0$ for any $t\geq0$ and
\begin{eqnarray*}
 \<\Cov(X(1))a\vert b\> &=& \E(\<X(1))\vert a\>\< X(1)\vert b\>) \\
 &=& \sum_{j=1}^d\E(\Theta_j(1))\<\Cov(L_j(1)) a_j\vert b_j\> \\
 &=& \Big\<\E(\Theta(1))\Cov(L(1))a \Big\vert b \Big\>
\end{eqnarray*}
for any $a,b\in H$.
\end{proof}

If $L$ is square integrable but not a martingale, then $\Theta$ has to be square integrable in order to ensure that $X$ is square integrable. Theorem \ref{S:charakterisierung quadint} below will show that square integrability of $\Theta$ is essentially necessary to ensure square integrability of $X$.
\begin{lem}\label{L:Erw und Varianz}
 Let $L$ and $\Theta$ be square integrable. Then $X$ is square integrable,
\begin{eqnarray*}
\E(X(1)) &=& \E\Theta(1)\E L(1)\quad\text{and}\\
\Cov(X(1)) &=& \E(\Theta(1))\Cov(L(1)) + \sum_{i,j=1}^d\Cov(\Theta(1))_{i,j} (\E L_i(1))\otimes(\E L_j(1))
\end{eqnarray*}
where $x\otimes y:H\rightarrow H,z\mapsto \<x\vert z\>y$ for any $x,y\in H$.
\end{lem}
\begin{proof}
 This follows easily by conditioning on the process $\Theta$.
\end{proof}

We can now state the characterisation of square integrability of $X$.
\begin{thm}\label{S:charakterisierung quadint}
 $X$ is square integrable if and only if $X_j$ is square integrable for all $j=1,\dots,d$. Let $j\in\{1,\dots,d\}$. Then $X_j$ is square integrable if and only if any of the following statements hold.
\begin{enumerate}
 \item $L_j$ and $\Theta_j$ are square integrable.
 \item $L_j$ is mean zero and square integrable and $\Theta_j$ is integrable.
 \item $\Theta_j =0$ a.s.
 \item $L_j=0$ a.s.
\end{enumerate}
Moreover, $X_j$ is mean zero and square integrable if and only if (2), (3) or (4) holds. If (1) holds for any $j=1,\dots,d$, then 
\begin{eqnarray*}
\E(X(1)) &=& \E\Theta(1)\E L(1)\quad\text{and}\\
\Cov(X(1)) &=& \E(\Theta(1))\Cov(L(1)) + \sum_{i,j=1}^d\Cov(\Theta(1))_{i,j} (\E L_i(1))\otimes(\E L_j(1))
\end{eqnarray*}
where $x\otimes y:H\rightarrow H,z\mapsto \<x\vert z\>y$ for any $x,y\in H$.
If (2) holds for any $j=1,\dots,d$, then
\begin{eqnarray*}
\Cov(X(1)) &=& \Cov(L(1))\E(\Theta(1)). 
\end{eqnarray*}
\end{thm}
\begin{proof}
The first statement follows directly from the equation $\vert X(1)\vert^2 = \sum_{j=1}^d \vert X_j(1)\vert^2$. The formulas at the end of the Theorem follow from the two previous Lemmas.  For the characterisation of square integrability of $X_j$ we can assume w.l.o.g.\ that $d=1$.

The if part follows from the two previous Lemmas. Assume that $X$ is square integrable. Let $(\beta,\Gamma,\mu)$ be the characteristics of $X$ as given in Theorem \ref{S:subordiniertes tripel}. Proposition \ref{P:Quadratstruktur} yields
\begin{eqnarray*}
  \infty &>& \int_H \vert x\vert^2 \mu(dx) \\
   &=& \int_{H} \vert x\vert^2 (a_0\nu)(dx) + \int_0^\infty\E(\vert L(\theta)\vert^2)F(d\theta).
\end{eqnarray*}
Thus $\int_0^\infty\E(\vert L(\theta)\vert^2)F(d\theta)<\infty$ and $\int_{H} \vert x\vert^2 (a_0\nu)(dx)<\infty$.

{\em Case 1:} $L$ is not square integrable. Then Proposition \ref{P:Quadratstruktur} implies that $\int_{H} \vert x\vert^2 \nu(dx)=\infty$. Thus $F=0$ and $a_0=0$. Hence $\Theta=0$ a.s.\ which is statement (3).

{\em Case 2:} $L$ is square integrable. Let $v:=\E(\vert L(1)-\E L(1)\vert^2)$ and $m:=\E L(1)$. Then $\E(\vert L(\theta)\vert^2) = \theta v+\theta^2\vert m\vert^2$ for any $\theta\geq0$. Hence we have
\begin{eqnarray*}
 \int_0^\infty \theta vF(d\theta)&<&\infty\quad\text{and} \\
 \int_0^\infty \theta^2\vert m\vert^2F(d\theta)&<&\infty,
\end{eqnarray*}

{\em Case 2.1:} $m\neq0$. Then $\int_0^\infty \theta^2F(d\theta)<\infty$. Hence \cite[Corollary 25.8]{sato.99} yields that $\Theta$ is square integrable.

{\em Case 2.2:} $m=0,v\neq0$. Then $L$ is mean zero and $\int_0^\infty \theta F(d\theta)<\infty$. Hence \cite[Corollary 25.8]{sato.99} yields that $\Theta$ is integrable. Thus we have statement (2).

{\em Case 2.3:} $m=0,v=0$. Since $0=v=\E(\vert L(1)\vert^2)$ we have $L=0$ a.s.

Lemma \ref{L:Quadrat Martingal} yields that if (2), (3) or (4) holds, then $X$ is mean zero and square integrable. If $X$ is mean zero and square integrable and (1) holds, then we have
$$0 = \E(X(1)) = \E\Theta(1)\E L(1).$$
Thus $\E\Theta(1)=0$ which yields (3) or $\E L(1)=0$ which implies (2).
\end{proof}

Theorem \ref{S:charakterisierung quadint} above is a complete characterisation of the second order structure of the L\'evy process $X$. However, there are L\'evy processes without finite second moment (e.g.\ see Theorem \ref{S:asLP} below). In that case the first order structure and the martingale property are still interesting. We now develop necessary and sufficient conditions for the existence of a first moment (cf.\ Theorem \ref{S:charakterisierung von Integrierbarkeit}) and we give a condition that suffices to show that $X$ is a martingale. Corollary \ref{C:int im Guasschen fall} is a restatement of Theorem \ref{S:charakterisierung von Integrierbarkeit} for the special case that $L$ is a Brownian motion without drift.

\begin{lem}\label{L:erstes Moment}
 Let $L$ and $\Theta$ be integrable. Then $X$ is integrable and $$\E X(1)=\E\Theta(1)\E L(1).$$
\end{lem}
\begin{proof}
 If $X$ is integrable, then conditioning on $\Theta$ yields the formula.

Let $f$ be the growth function of $L$, cf.\ Definition \ref{D:growth function}. Then Lemma \ref{L:schranke der Wachstumsfunktion} yields that there is $C>0$ and $f(\theta)\leq 1+C\vert \theta\vert$ for any $\theta\in\mathbb R_+^d$. Lemma \ref{L:X integrierbar} yields that $X$ is integrable if $f(\Theta(1))$ is integrable. However,
$$\E f(\Theta(1))\leq 1+C\E\vert\Theta(1)\vert < \infty.$$
\end{proof}

We have seen that the martingale property of $L$ allows to put weaker assumptions on $\Theta$ to ensure that $X$ is square integrable. If $L$ is a square integrable martingale, then similar as before a weaker assumption than in Lemma \ref{L:erstes Moment} on $\Theta$ is sufficient to ensure integrability of $X$.
\begin{thm}\label{S:X Martingal}
 Let $L$ be a square integrable martingale and assume that $\sqrt{\vert\Theta(1)\vert}$ is integrable (or equivalently $\int_{\vert\theta\vert>1}\sqrt{\vert \theta\vert}F(d\theta)<\infty$). Then $X$ is integrable and mean zero.
\end{thm}
\begin{proof}
 Proposition \ref{P:integrierbarkeit} yields $\E(\sqrt{\vert\Theta(1)\vert})<\infty$ if and only if $\int_{\vert\theta\vert>1}\sqrt{\vert\theta\vert}F(d\theta)<\infty$.

 Let $f$ be the growth function of $L$ in the sense of Defintion \ref{D:growth function}. Then Lemma \ref{L:schranke der Wachstumsfunktion Martingalfall} yields that there is $C>0$ such that $f(\theta)\leq C\sqrt{\vert\theta\vert}$ for any $\theta\in\mathbb R_+^d$. Thus
$$\E f(\Theta(1))\leq C\E\sqrt{\vert\Theta(1)\vert}<\infty.$$
Lemma \ref{L:X integrierbar} yields that $X$ is integrable. Moreover, $\E(X(1)\vert \Theta) = \Theta(1) \E(L(1)) =0$. Thus $X$ is mean zero.
\end{proof}

\begin{thm}\label{S:charakterisierung von Integrierbarkeit}
Let $\Theta$ be non trivial, i.e.\ $P(\Theta\neq0)>0$. Then $X$ is integrable if and only if $L$ is integrable and $\int_{\vert\theta\vert>1} \E(\vert L(\theta)\vert)F(d\theta) <\infty$. If $L$ and $\Theta$ (and hence $X$) are integrable, then
$$\E X(1) =\E L(1)\E\Theta(1).$$
If $X_j$ is integrable but $\Theta_j$ is not integrable, then $X_j$ is mean zero where $j\in\{1,\dots,d\}$.
\end{thm}
\begin{proof}
 Let $(\beta,\Gamma,\mu)$ be the characteristics of $X$ as given in Theorem \ref{S:subordiniertes tripel} and let $f$ be the growth function of $L$, cf.\ Definition \ref{D:growth function}. We have
\begin{eqnarray*}
 && \int_{\{\vert x\vert>1\}}\vert x\vert \mu(dx) 
   = \int_{\{\vert x\vert>1\}}\vert a_0x\vert \nu(dx) + \int_{\mathbb R_+^d} \E(\vert L(\theta)\vert1_{\vert L(\theta)\vert>1})F(d\theta).
\end{eqnarray*}

$\Rightarrow:$ Let $X$ be integrable. Then Proposition \ref{P:integrierbarkeit} yields $\int_{\{\vert x\vert>1\}}\vert x\vert \mu(dx)<\infty$. Thus Proposition \ref{P:integrierbarkeit} implies that $L$ is integrable.

$\Leftarrow:$  Let $L$ be integrable and $\int_{\vert\theta\vert>1} \E(\vert L(\theta)\vert)F(d\theta) <\infty$. Then
$$\int_{\vert \theta\vert>1} f(\theta)F(d\theta) = \int_{\vert\theta\vert>1}\E\vert L(\theta)\vert F(d\theta)<\infty$$
where $f$ denotes the growth function of $L$. Proposition \ref{P:integrierbarkeit} yields $\E f(\Theta(1))<\infty$. Lemma \ref{L:X integrierbar} yields the first claim.

Lemma \ref{L:erstes Moment} yields the formula for the moment above.

Now let $j\in\{1,\dots,d\}$ and assume that $X_j$ is integrable but $\Theta_j$ is not. We have already shown that this implies that $L_j$ is integrable. Let $g(\theta):=\E L_j(\theta)=\theta\E L_j(1)$ for any $\theta\in\mathbb R_+$. Thus
$$\E X_j(1) = \E(\E\left( L_j(\Theta_j(1))\vert \Theta\right)) = \E g(\Theta(1)) =\E\big(\Theta_j(1)\E L_j(1)\big).$$
We see from that equation that $\Theta_j(1)\E L_j(1)$ is integrable. Since $\Theta_j(1)$ is not integrable we conclude that $\E L_j(1)=0$. Hence $\E X_j(1)=0$.
\end{proof}

We are especially interested in the case that $L$ is a Gaussian L\'evy process (i.e.\ a cylindrical Brownian motion). Then the martingale property of $X$ can be characterised easily in terms of a moment condition of $\Theta$. 
\begin{cor}\label{C:int im Guasschen fall}
 Assume that $L$ is Gaussian and mean zero and assume that $\tr(Q_j)\neq0$ for all $j\in\{1,\dots,d\}$. Then $X$ is integrable if and only if $\sqrt{\vert\Theta(1)\vert}$ is integrable. In that case $X$ is mean zero.
\end{cor}
\begin{proof}
 Theorem \ref{S:X Martingal} implies the if part and the last statement. Let $X$ be integrable. Theorem \ref{S:charakterisierung von Integrierbarkeit} yields $\int_{\vert\theta\vert>1} \E\vert L(\theta)\vert F(d\theta) <\infty$. We also have
\begin{eqnarray*}
\E\vert L(\theta)\vert &\geq& \frac{1}{\sqrt{d}}\sum_{j=1}^d \E\vert L_j(\theta_j)\vert \\
  &=& \frac{1}{\sqrt{d}}\sum_{j=1}^d \theta_j^{1/2}\E\vert L_j(1)\vert \\
 &\geq& \frac{1}{\sqrt{d}}\sqrt{\vert \theta\vert}\min\{\E\vert L_j(1)\vert:j=1,\dots,d\} \\
 &=& \sqrt{\vert\theta\vert}C
\end{eqnarray*}
for any $\theta\in\mathbb R_+^d$ where $C:=\frac{\min\{\E\vert L_j(1)\vert:j=1,\dots,d\}}{\sqrt{d}}>0$. Hence 
$$\int_{\vert\theta\vert>1}\sqrt{\vert\theta\vert}F(d\theta) <\infty.$$
Proposition \ref{P:integrierbarkeit} yields $\sqrt{\vert\Theta(1)\vert}$ is integrable.
\end{proof}
Gaussian L\'evy processes $L$ will play a main role when defining some explicit classes of subordinated L\'evy processes,
which is the topic of the next Section.

\section{Examples and application}
In this Section we construct three classes of subordinated L\'evy processes, extending the popular uni/multi-variate normal inverse 
Gaussian, $\alpha$-stable and variance Gamma L\'evy processes. 

\subsection{Hilbert space valued normal inverse Gaussian process}\label{A:HNIG}
Multivariate normal inverse Gaussian distributions (MNIG-distributions) have been first introduced in \cite{rydberg.97}. These distributions can, of course, also be generated from a multivariate Brownian motion and an inverse Gaussian process by subordination, i.e.\ the subordinated Brownian motion is a process where its marginal distributions are MNIG. We generalise this approach to construct Hilbert space-valued normal inverse Gaussian (HNIG) processes.

\begin{defn}
 A L\'evy process $Y$ is an {\em HNIG-process} if there are $s,c\in\mathbb R_+$, $b\in H$ and a positive semi-definite trace class operator $Q$ on $H$ such that its L\'evy exponent is given by
$$\rho: H\rightarrow\mathbb C, u\mapsto s\left(c-\sqrt{c^2+\<Qu\vert u\>-i2\<u\vert b\>}\right)$$
where $\sqrt{\cdot}$ denotes the main branch of the root function. Here, $(s,c,b,Q)$ are the {\em parameters of the HNIG-process} $Y$. A {\em degenerate HNIG-process} is an HNIG-process where its second parameter is $0$, i.e.\ there are $s\in\mathbb R_+$, $b\in H$ and a positive semi-definite trace class operator $Q$ on $H$ such that $(s,0,b,Q)$ are the parameters of $Y$.
\end{defn}

Let us start with the construction of non-degenerate HNIG processes and discuss some of their properties.
\begin{thm}\label{S:HNIG}
 Let $s,c\in\mathbb R_+$, $c\neq0$, $b\in H$ and $Q$ a positive semi-definite trace class operator on $H$. Then there is an HNIG-process $Y$ with parameters $(s,c,b,Q)$. The characteristics $(\beta,\Gamma,\mu)$ of $Y$ are given by
\begin{eqnarray*}
 \beta &=& \frac{sb}{c}-\int_{\vert x\vert>1}x\mu(dx),\\
 \Gamma &=& 0 \quad\text{and}\\
 \mu(A) &=& \int_0^\infty \Phi_\theta(A) \frac{s}{\sqrt{2\pi \theta^3}}e^{-c^2\theta/2}d\theta
\end{eqnarray*}
for any Borel set $A\subseteq H$ where $\Phi_\theta$ denotes the Gaussian measure on $H$ with mean $\theta b$ and covariance operator $\theta Q$. Moreover, $\E Y(1) = \frac{s}{c}b$ and $\Cov Y(1) = \frac{s b\otimes b}{c^3}+\frac{s}{c}Q$. If $b=0$, then $\beta=0$ and $Y$ is symmetric. The distribution of $TY(t)$ is MNIG in the sense of \cite[Section 10.5]{barndorff.shephard.12} for any bounded linear operator $T$ from $H$ to $\mathbb R^n$ and any $n\in\mathbb N$.
\end{thm}
\begin{rem}
 In the theorem above the requirement $c\neq0$ is not needed to ensure existence. However, the resulting degenerate HNIG-process will behave differently, cf.\ Proposition \ref{P:degeneriert Cauchy} below.
\end{rem}
\begin{proof}[Proof of Theorem \ref{S:HNIG}]
 Let $L$ be a Brownian motion with drift $b$ and covariance operator $Q$. Let $\Theta$ be an inverse Gaussian process with parameters $s,d$, i.e.\ it is a pure-jump subordinator and its L\'evy measure is given by
$$ F(d\theta) = \frac{s}{\sqrt{2\pi \theta^3}}e^{-c^2\theta/2}1_{\{\theta>0\}}d\theta,$$
cf.\ \cite[Example 7.25]{barndorff.shephard.12}.
Then its Laplace exponent is given by
$$\psi:\mathbb R_-+i\mathbb R\rightarrow\mathbb C, v\mapsto s\left(c-\sqrt{c^2-2v}\right)$$
where $\sqrt{\cdot}$ denotes the main branch of the root function. Theorem \ref{S:Fourier-Darstellung} yields that the L\'evy exponent of the L\'evy process $X(t):=L(\Theta(t))$ is 
$$\rho: H\rightarrow\mathbb C, u\mapsto s\left(c-\sqrt{c^2+\<Qu\vert u\>-i2\<u\vert b\>}\right).$$
Theorem \ref{S:subordiniertes tripel} yields that $(\beta,\Gamma,\mu)$ as defined as above is the characteristics of $X$. Theorem \ref{S:charakterisierung quadint} implies that $X$ is square integrable and that its expectation and its covariance operator are given as above. Let $n\in\mathbb N$, $t\in\mathbb R_+$ and $T:H\rightarrow\mathbb R^n$ be bounded and linear. Then $T(X(t)) = (T\circ L)(\Theta(t))$. $W:=(T\circ L)$ is a Gaussian process on $\mathbb R^n$ with drift $Tb$ and covariance operator $TQT^*$ where $T^*$ denotes the dual operator of $T$. Hence $TX(t) = W(\Theta(t))$ and consequently its distribution is MNIG, cf.\ \cite[Section 10.5]{barndorff.shephard.12}.

Let $Y$ be any HNIG-process with parameters $(s,c,b,Q)$. Then Remark \ref{R:characteristische Funktion} yields that $X$ and $Y$ have the same distribution and hence they have the same moments. Since $X$ and $Y$ have the same characteristic function they have the same characteristics.
\end{proof}

In order to construct degenerate HNIG-processes we use a different subordinator, name\-ly $0.5$-stable subordinator. Subordination of Brownian motion with an $\alpha$-stable subordinator will be investigated in section \ref{A:alpha stabil} in more detail.
\begin{prop}\label{P:degeneriert Cauchy}
 Let $s\in\mathbb R_+$, $b\in H$ and $Q$ a positive semi-definite trace class operator on $H$. Then there is an HNIG-process $C$ with parameters $(s,0,b,Q)$. $C$ is not integrable.
\end{prop}
\begin{proof}
 Let $L$ be a Brownian motion with drift $b$ and covariance operator $Q$. Let $s\geq0$ and $\Theta$ be the $0.5$-stable subordinator with L\'evy measure
$$F(d\theta) = s\frac{\theta^{-1.5}}{\Gamma(-0.5)}d\theta.$$
Then its Laplace exponent is given by
$$\psi:\mathbb R_-+i\mathbb R\rightarrow\mathbb C, v\mapsto \begin{cases}s\exp(-0.5\Log(-v))&v\neq0,\\0&v=0\end{cases}$$
where $\Log$ denotes the main branch of the logarithm. Theorem \ref{S:Fourier-Darstellung} yields that the L\'evy exponent of the L\'evy process $X(t):=L(\Theta(t))$ is given by
$$\rho(u) = \psi(\phi(u)) = s\sqrt{-\phi(u)}$$
as desired. Observe that $\sqrt{\Theta(1)}$ is not integrable. Hence Corollary \ref{C:int im Guasschen fall} yields the claim if $b=0$ and Theorem \ref{S:charakterisierung von Integrierbarkeit} yields the claim if $b\neq0$.
\end{proof}

\begin{prop}
 Let $Y$ be a process on $H$. Then $Y$ is a HNIG-process if and only if $TY$ is an MNIG-process for every finite dimensional operator $T$ on $H$.
\end{prop}
\begin{proof}
 This can be simply read from the characteristic function.
\end{proof}

\subsection{$\alpha$-stable Hilbert space valued L\'evy processes}\label{A:alpha stabil}
Stable L\'evy processes have been studied extensively and used in mathematical finance. We refer the reader to the book of Sato~\cite[Chapter 3]{sato.99} for reference. In this section we will investigate some properties of symmetric stable L\'evy processes and construct some of them, see Theorem \ref{S:asLP} below. Here again we make use of subordination and generate them from Brownian motion. Like the finite dimensional case integrability properties of symmetric stable L\'evy processes are related to the index of the process. Many other properties can be derived as in the finite dimensional case, cf.\ Sato~\cite[Chapter 3]{sato.99}.

We also want to point out that CGMY processes (cf.\ \cite{madan.yor.08}) can be constructed by subordinating a Brownian motion with drift with an $\alpha$-stable subordinator. This can be easily generalised to subordination of Hilbert space valued Brownian motions.

Let us first recall the definition of strictly $\alpha$-stable processes.
\begin{defn}
 Let $\alpha\in\mathbb R_+$. A stochastic process $Y$ is a {\em strictly $\alpha$-stable process} if $Y(t^\alpha)$ and $tY(1)$ have the same distribution for any $t\in\mathbb R_+$.
\end{defn}

An explicit construction of stable Hilbert space valued L\'evy processes has been taken out in \cite[Example 4.38]{peszat.zabczyk.07}. We make use of this construction and discuss some properties of them.
\begin{thm}\label{S:asLP}
For each $\alpha\in(0,2]$ and each positive semi-definite trace class operator $Q\neq0$ on $H$ there is a symmetric $H$-valued strictly $\alpha$-stable L\'evy process $Y$ with L\'evy exponent
$$\rho:H\rightarrow\mathbb C,u\mapsto -\<Qu\vert u\>^{\alpha/2}.$$
Such an strictly $\alpha$-stable process is square integrable if and only if $\alpha=2$ and it is integrable and mean zero if and only if $\alpha>1$.

Let $Y$ be a symmetric strictly $\alpha$-stable L\'evy process which is non-trivial, i.e.\ $P(Y\neq0)>0$. Then $\alpha\in(0,2]$. Moreover, there is a symmetric continuous function $f:S_H\rightarrow \mathbb R_+$ such that the characteristic exponent of $Y$ is given by
$$\rho:H\rightarrow\mathbb C,u\mapsto -\vert u\vert^\alpha f\left(\frac{u}{\vert u\vert}\right)$$
where $S_H:=\{x\in H:\vert x\vert=1\}$ denotes the sphere in $H$.
\end{thm}
\begin{proof}
Let $\alpha\in(0,2)$ and $Q$ be a positive definite trace class operator on $H$. Let $L$ be a mean zero Gaussian L\'evy process with covariance operator $2Q$. Let $\Theta$ be an $\alpha/2$-stable subordinator. Then its Laplace exponent is given by
$$\psi:\mathbb R_-+i\mathbb R\rightarrow\mathbb C,s\mapsto \begin{cases}-\exp(\alpha/2\Log(-s))&\text{ if }s\neq0,\\0&\text{otherwise}\end{cases}$$
where $\Log$ denotes the main branch of the logarithm (cf.\ \cite[page 73]{bertoin.96}). Hence Theorem \ref{S:Fourier-Darstellung} yields that the characteristic function of $X(t):=L(\Theta(t))$ is given by
$$\rho:H\rightarrow\mathbb C,u\mapsto -\<Q u\vert u\>^{\alpha/2}.$$
We have $X(t^\alpha) = L(t^{\alpha/2}\Theta(1)) = tX(1)$ for any $t\in\mathbb R_+$. Hence $X$ is a symmetric $H$-valued strictly $\alpha$-stable L\'evy process. If $\alpha=2$, then $X=L$ and hence it is square integrable. Theorem \ref{S:charakterisierung quadint} implies that $X$ is not square-integrable if $\alpha\neq2$. Corollary \ref{C:int im Guasschen fall} yields that $X$ is integrable if and only if $\alpha>1$.

Now let $\alpha\in\mathbb R_+$ be arbitrary and $Y$ be a strictly $\alpha$-stable non-trivial L\'evy process. Then there is $u\in H$ such that $P(\<u\vert Y\>=0)\neq0$. \cite[Theorem 13.15]{sato.99} applied to the strictly $\alpha$-stable process $\<u\vert Y\>$ yields that $\alpha\in(0,2]$. Let $\rho$ be the L\'evy exponent of $Y$ and define $f:=-\rho\vert_{S_H}$. Let $u\in H\backslash\{0\}$ and define $t:=\vert u\vert$ and $v:=u/t\in S_H$. Then
$$\exp(\rho(u)) = Ee^{i\< u\vert Y(1)\>} = Ee^{i\< tv\vert Y(1)\>} = E^{i\< v\vert Y(t^{\alpha})\>} = \exp(t^\alpha f(v)).$$
Thus $\rho(u) = t^\alpha f(v)$. Since $Y$ is symmetric $\rho$ is real valued and so is $f$. $\mathrm{Re}(\rho)$ is bounded by $0$ because the characteristic function of $Y$ is bounded by $1$. Hence $f(v)\in\mathbb R_+$ for any $v\in S_H$. Symmetry of $f$ follows from symmetry of $\rho$ which follows from symmetry of $Y$.
\end{proof}

\begin{prop}\label{P:asLP Wachstumsfunktion}
 Let $\alpha\in(1,2)$ and $L$ be an integrable strictly $\alpha$-stable L\'evy process such that $L$ is non-trivial, i.e.\ $P(L\neq0)>0$. Then $X$ is integrable if and only if $\vert\Theta(1)\vert^{1/\alpha}$ is integrable.
\end{prop}
\begin{proof}
 Let $f$ be the growth function of $L$. Then $f(\theta) = \E \vert L(\theta)\vert = \sum_{j=1}^d\theta_j^{1/\alpha}\E\vert L_j(1)\vert$. Thus theorem \ref{S:charakterisierung von Integrierbarkeit} yields that $X$ is integrable if and only if $\int_{\mathbb R_+^d}\vert\theta\vert F(d\theta)<\infty$. Proposition \ref{P:integrierbarkeit} implies the claim.
\end{proof}

\subsection{Hilbert space valued variance Gamma process}
Variance Gamma processes have been introduced by Madan and Seneta~\cite{madan.senata.90} and a multivariate version have been introduced by the same authors. Since their introduction, they have been used extensively in financial modelling (see e.g.\ \cite{madan.al.98}). Univariate Variance Gamma processes can be constructed as a difference of two independent Gamma processes or by subordinating a Brownian motion with a Gamma process. The latter approach can be easily generalised to Hilbert space valued L\'evy processes which we do in this section. Theorem \ref{S:HVG} below contains an analysis of Hilbert space valued Variance Gamma processes (HVG) and a construction of those processes is taken out in the proof of this theorem.

\begin{defn}
 A L\'evy process $Y$ is a {\em Hilbert space valued variance gamma process} or {\em HVG-process} if there are $a\in\mathbb R_+$, $b\in H$ and a positive semi-definite trace class operator $Q$ on $H$ such that its L\'evy exponent is given by
 $$\rho: H\rightarrow\mathbb C, u\mapsto ,a\Log(1+1/2\<Qu\vert u\>-i\<b\vert u\>)$$
where $\Log$ denotes the main branch of the logarithm. $(a,b,Q)$ are the {\em parameters of the HVG-process} $Y$
\end{defn}

\begin{thm}\label{S:HVG}
 Let $a\in\mathbb R_+$, $b\in H$ and $Q$ a positive semi-definite trace class operator on $H$. Then there is an HVG-process $Y$ with parameters $(a,b,Q)$. The characteristics $(\beta,\Gamma,\mu)$ of $Y$ are given by
\begin{eqnarray*}
 \beta &=& ab-\int_{\vert x\vert>1}x\mu(dx), \\
 \Gamma &=& 0 \quad\text{and}\\
 \mu(A) &=& \int_0^\infty \Phi_t(A) at^{-1}e^{-t}dt
\end{eqnarray*}
for any Borel set $A\subseteq H$ where $\Phi_t$ denotes the Gaussian measure on $H$ with mean $tb$ and covariance operator $tQ$. Moreover, $\E Y(1) = ab$ and $\Cov Y(1) = a b\otimes b+aQ$. $\<u\vert Y\>$ is a variance gamma process for any $u\in H$. If $b=0$, then $\beta=0$ and $Y$ is symmetric.
\end{thm}
\begin{proof}
 Let $L$ be a Brownian motion with drift $b$ and covariance operator $Q$. Let $\Theta$ be a gamma process with parameters $(a,1)$, i.e.\ it is a pure-jump subordinator and its L\'evy measure is given by
$$ F(d\theta) = a\theta^{-1}e^{-\theta}1_{\{\theta>0\}}d\theta,$$
cf.\ \cite[page 73]{bertoin.96}.
Then its Laplace exponent is given by
$$\psi:\mathbb R_-+i\mathbb R\rightarrow\mathbb C, v\mapsto a\Log(1-s)$$
where $\Log$ denotes the main branch of the logarithm. Theorem \ref{S:Fourier-Darstellung} yields that the L\'evy exponent of the L\'evy process $X(t):=L(\Theta(t))$ is 
$$\rho: H\rightarrow\mathbb C, u\mapsto ,a\Log(1+1/2\<Qu\vert u\>-i\<b\vert u\>).$$
Theorem \ref{S:subordiniertes tripel} yields the specific form of the characteristics of $X$. Theorem \ref{S:charakterisierung quadint} yields that $X$ is square integrable and that its expectation and its covariance operator are given as above. Let $u\in H$. Then $\<u\vert X(t)\> = \<u\vert L\>(\Theta(t)$. $W:=\<u\vert L\>$ is a Gaussian L\'evy process on $\mathbb R$ with drift $\<b\vert u\>$ and covariance $\<Qu\vert u\>$. Hence $\<u\vert X(t)\> = W(\Theta(t))$ and consequently its a variance gamma process.

Let $Y$ be any HVG-process with parameters $(a,b,Q)$. Then Remark \ref{R:characteristische Funktion} yields that $X$ and $Y$ have the same distribution and hence they have the same moments. Since $X$ and $Y$ have the same characteristic function they have the same characteristics.
\end{proof}

\appendix
\section{}
\subsection{Properties of Hilbert space valued L\'evy processes}
\begin{lem}\label{L:Eindeutigkeit des Tripel}
 Let $Y$ be an $H$-valued process. Let $(b_1,Q_1,\nu_1)$ and $(b_2,Q_2,\nu_2)$ both be characteristics of $Y$. Then $b_1=b_2$, $Q_1=Q_2$ and $\nu_1=\nu_2$. In other words, the process $Y$ has exactly one characteristic.
\end{lem}
\begin{proof}
Define
$$\phi_k:H\rightarrow\mathbb C,u\mapsto i\<b_k\vert u\>-\frac{1}{2}\<Q_ku\vert u\>+\int_H(e^{i\<u\vert x\>}-1-i\<u\vert x1_{\{\vert x\vert\leq1\}})\nu(dx)$$
for $k\in\{1,2\}$. \cite[Theorem 4.27]{peszat.zabczyk.07} yields
$$\exp(\phi_1(u))  = \E(e^{i\<u\vert Y(1)\>}) = \exp(\phi_2(u))$$
for any $u\in H$. In particular, we have $\phi_1=\phi_2$.

Let $u\in H$ and define $\kappa:\mathbb R\rightarrow H,t\mapsto tu$ and $\psi:=\phi_1\circ\kappa=\phi_2\circ\kappa$. Then $\psi$ is the L\'evy exponent of the $\mathbb R$-valued L\'evy process $\<Y\vert u\>$. \cite[Lemma II.2.44]{js.87} yields $\<b_1\vert u\>=\<b_2\vert u\>$, $\<Q_1u\vert u\>=\<Q_2u\vert u\>$ and $\nu_1^{\<u\vert\cdot\>}=\nu_2^{\<u\vert\cdot\>}$. Since this is true for any $u\in H$ we have $b_1=b_2$. $Q_1,Q_2$ are positive matrices and hence 
$$\<Q_1u\vert v\>= \frac{1}{2}\left(\<Q_1u\vert u\>+\<Q_1v\vert v\>\right)=\<Q_2u\vert v\>$$
for any $u,v\in H$. Hence $Q_1=Q_2$. \cite[page 38]{ledoux.talagrand.10} yields that the Borel-$\sigma$-algebra on $H$ coincides with the cylindrical $\sigma$-algebra, i.e.\ the $\sigma$-algebra generated by the continuous linear functionals. Since $\nu_1^{\<u\vert\cdot\>}=\nu_2^{\<u\vert\cdot\>}$ for any $u\in H$ they coincide on the cylindrical $\sigma$-algebra. Thus $\nu_1=\nu_2$.
\end{proof}

\begin{defn}\label{D:submultiplicative}
 A function $g:H\rightarrow\mathbb R_+$ is called {\em  submultiplicative} if $g(x+y)\leq ag(x)g(y)$ for some constant $a>0$, cf.\ \cite[Definition 25.2]{sato.99}. A function $f:H\rightarrow\mathbb R_+$ is called {\em subadditive} if $f(x+y)\leq f(x)+f(y)$.
\end{defn}

\begin{lem}\label{L:Schranke fuer g}
 Let $g:H\rightarrow\mathbb R_+$ be a submultiplicative function which is bounded on a neighbourhood of zero. Then there are $c_1>0$ and $c_2\in\mathbb R$ such that $g(x)\leq c_1\exp(c_2\vert x\vert)$ for all $x\in H$.
\end{lem}
\begin{proof}
 This proof is along the lines of \cite[Lemma 25.5]{sato.99}. W.l.o.g.\ assume that $a\geq1$. Let $y\in H$ and $n\in\mathbb N$ then applying submultiplicativity $n$ times we get
$$g(yn) \leq a^{n-1}g(y)^n.$$

If $g(0)=0$, then $g(y)\leq ag(y)g(0)=0$ for any $y\in U$ and the claim follows. W.l.o.g.\ we may assume that $g(0)\neq0$. Let $\epsilon>0$ such that $g$ is bounded by $c_1\geq1$ on the set $\{x\in H:\vert x\vert<\epsilon\}$. Let $y\in H$ and $n\in\mathbb N$ such that $\vert y\vert/n\leq\epsilon\leq \vert y\vert/(n-1)$. Then
\begin{eqnarray*}
 g(y) &\leq& a^{n-1}(g(y/n))^{n} \\
  &\leq& (ac_1)^{\vert y\vert/\epsilon}c_1 \\
  &=& c_1\exp(\vert y\vert\log(ac_1)/\epsilon).
\end{eqnarray*}
The Lemma follows.
\end{proof}

\begin{rem}
 Let $\alpha,\beta$ be finite Borel measures on $H$. Recall that the convolution $\alpha*\beta$ of the measures $\alpha,\beta$ is a finite Borel measure on $H$ which is defined by
$$(\alpha*\beta)(B) := \int_H\beta(B-x)\alpha(dx)\quad B\in\mathcal B(H)$$
for any $B\in\mathcal B(H)$. Moreover, $\alpha^{*0}$ is definedy to be the dirac-measure in $0$ and $\alpha^{*n+1}:=\alpha*(\alpha^{*n})$ for all $n\in\mathbb N$. The total mass of the measure $\alpha^{*n}$ is given by
$$\alpha^{*n}(H) = (\alpha(H))^n.$$
Let $t\geq0$. Then $\gamma_n:=\sum_{k=0}^n\frac{(t\alpha)^{*n}}{n!}$ converges w.r.t.\ the total variation norm on the space of signed measures of finite total variation to a measure which we denote by $\exp(t\alpha)$. The formula above yields $(\exp(t\alpha))(H) = \exp(t\alpha(H))$ and hence {\em the convolution semigroup generated by $\alpha$} which is given by
$$\mu_t:=\exp(-t\alpha(H))\exp(t\alpha)\quad t\geq0$$
 is a probability measure such that $\mu_t*\mu_s=\mu_{t+s}$ for any $s,t\geq0$.
\end{rem}

\begin{lem}\label{L:nu und mu bei Poisson}
 Let $\nu_1$ be a finite Borel measure on $H$ and define $$\mu_t:=\exp(-t\nu_1(H))\exp(t\nu_1)$$ for any $t\in\mathbb R_+$. Let $g$ be submultiplicative with constant $a$. Then
$$\exp(-t\nu_1(H))\int_H g(x)\nu_1(dx) \leq \int_H g(x)\mu_t(dx) \leq \exp\left(ta\int_Hg(x)\nu_1(dx)-t\nu(H)\right)$$
for any $t>0$. In particular, $\int_H g(x)\nu_1(dx)<\infty$ if and only if $\int_Hg(X)\mu_t(dx)<\infty$ for some (and hence all) $t>0$.
\end{lem}
\begin{proof}
 The first inequality is trivial. Let $t>0$. Then
$$\int_Hg(x)t^n\nu_1^n(dx)\leq a^{n-1}\left(t\int_Hg(x)\nu_1(dx)\right)^n\leq \left(ta\int_Hg(x)\nu_1(dx)\right)^n.$$
Thus 
$$\int_Hg(x)\mu_t(dx)\leq\exp\left(ta\int_Hg(x)\nu_1(dx)-t\nu_1(H)\right).$$
\end{proof}

\begin{lem}\label{L:endl bei beschr Spruengen}
 Assume that $L$ has bounded jumps. Let $g$ be submultiplicative and bounded on a neighbourhood of zero. Then
  $$Eg(L(t)) <\infty$$
for any $t\geq0$.
\end{lem}
\begin{proof}
 Lemma \ref{L:Schranke fuer g} yields that there are $c_1>0$, $c_2\in\mathbb R$ such that $g(x)\leq c_1\exp(c_2\vert x\vert)$. \cite[Theorem 4.4]{peszat.zabczyk.07} yields
$$Eg(L(t)) \leq c_1 \E(e^{c_2\vert L(t)\vert}) <\infty.$$
\end{proof}

The next Proposition does not assume any local boundedness as in \cite[Theorem 25.3]{sato.99}. However, it already follows from the Proof of \cite[Theorem 25.3]{sato.99} that the boundedness is not needed for the next Proposition in the finite dimensional case.
\begin{prop}\label{P:endl Erw endl levy mass}
 Let $t>0$ and $g$ be submultiplicative and measurable and assume that $Eg(L(t)) < \infty$. Then $\int_{\{\vert x\vert>1\}}g(x)\nu(dx) <\infty$.
\end{prop}
\begin{proof}
This proof is along the lines of the proof of \cite[Theorem 25.3]{sato.99}.

\cite[Theorem 4.23]{peszat.zabczyk.07} yields that $L=L_1+L_2$ where $L_1$ and $L_2$ are independent L\'evy processes, $L_1$ has jumps bounded by $1$ and $L_2$ is a compound Poisson process with L\'evy measure $\nu_2(B):=\nu(B\cap\{\vert x\vert>1\})$. Let $\mu_1$ be the distribution of $L_1(t)$ and $\mu_2$ be the distribution of $L_2(t)$. Since $L_2$ is a compound Poisson process we have $\mu_2 = \exp(-t\nu_2(H))\exp(t\nu_2)$. Moreover, we have
$$ \int_H\int_Hg(x+y)\mu_2(dy)\mu_1(dx) = Eg(L(t)) <\infty.$$
Thus there is $x\in H$ such that
$$\int_Hg(x+y)\mu_2(dy)<\infty.$$
Hence Lemma \ref{L:nu und mu bei Poisson} yields 
$$\int_{\{\vert x\vert>1\}}g(y)\nu(dy) = \int_Hg(y)\nu_2(dy) \leq ag(-x)\int_Hg(x+y)\nu_2(dy)<\infty.$$
\end{proof}

Now we generalise \cite[Theorem 25.3]{sato.99} to Hilbert space valued L\'evy processes.
\begin{thm}\label{S:g Moment}
 Let $g$ be submultiplicative, bounded and measurable on a neighbourhood of zero. Then $\int_{\{\vert x\vert>1\}}g(x)\nu(dx) <\infty$ if and only if $\E g(L(t))<\infty$ for some (and hence all) $t>0$.
\end{thm}
\begin{proof}
Proposition \ref{P:endl Erw endl levy mass} yields the only if part.

 \cite[Theorem 4.23]{peszat.zabczyk.07} yields that $L=L_1+L_2$ where $L_1$ and $L_2$ are independent L\'evy processes, $L_1$ has jumps bounded by $1$ and $L_2$ is a compound Poisson process with L\'evy measure $\nu_2(B):=\nu(B\cap\{\vert x\vert>1\})$. Moreover, $\E(g(L(t))) \leq \E g(L_1(t)) \E g(L_2(t))$ where the first factor is finite by Lemma \ref{L:endl bei beschr Spruengen}. Lemma \ref{L:nu und mu bei Poisson} yields that $\E g(L_2(t))<\infty$ because 
$$\int_Hg(x)\nu_2(dx) = \int_{\{\vert x\vert>1\}}g(x)\nu(dx) <\infty.$$
Thus $\E g(L(t)) <\infty$.
\end{proof}

\begin{defn}\label{D:growth function}
 The {\em growth function} of the process $L$ is the function
  $$f:\mathbb R^d_+\rightarrow\mathbb R_+\cup\{\infty\},t\mapsto \E\vert L(t)\vert.$$
\end{defn}

\begin{rem}\label{R:growth function}
Let $f$ be the growth function of $L$.

If $L$ is integrable, then 
\begin{itemize}
 \item $f$ is continuous,
 \item $f(\theta_1+\theta_2)\leq f(\theta_1)+f(\theta_2)$ for any $\theta_1,\theta_2\in\mathbb R^d_+$,
 \item $f(\theta_1)\leq f(\theta_1+\theta_2)$ for any $\theta_1,\theta_2\in\mathbb R^d_+$ and
 \item $f(\theta)<\infty$ for any $\theta\in\mathbb R_+$.
\end{itemize}

If $f_j$ is the growth function of $L_j$, then we have $f(\theta)\leq \sum_{j=1}^d f_j(\theta_j)\leq \sqrt{d}f(\theta)$ for any $\theta\in\mathbb R_+^d$.
\end{rem}

\begin{lem}\label{L:X integrierbar}
 Let $f$ be the growth function of $L$ and let $f(\Theta(1))$ be integrable. Then $X$ is integrable.
\end{lem}
\begin{proof}
 We have
\begin{eqnarray*}
 \E\vert X(1)\vert &=& \E\left(\E\big(\vert L(\Theta(1))\vert \big\vert \Theta\big)\right) \\
  &=& \E(f(\Theta(1))) \\
  &<& \infty.
\end{eqnarray*}
\end{proof}

\begin{prop}\label{P:integrierbarkeit}
 Let $f:H\rightarrow\mathbb R_+$ be bounded in a neighbourhood of zero and subadditive, i.e.\ $f(x+y)\leq f(x)+f(y)$. Then $f(L(t))$ is integrable for some (and hence all) $t>0$ if and only if $$\int_{\{\vert x\vert>1\}} f(x) \nu(dx)<\infty.$$

In particular, $\E\vert L(1)\vert<\infty$ if and only if $$\int_{\{\vert x\vert>1\}} \vert x\vert \nu(dx)<\infty.$$
\end{prop}
\begin{proof}
 Define $g(x):=2\vee f(x)$ for any $x\in H$. Let $x,y\in H$ such that $f(x)\leq f(y)$. Then
 \begin{eqnarray*}
  g(x+y) &\leq& g(x)+g(y) \\
       &\leq& 2g(y) \\
       &\leq& g(x)g(y).
 \end{eqnarray*}
 Thus $g$ is submultiplicative and bounded in a neighbourhood of zero. Theorem \ref{S:g Moment} yields $\E g(L(1))<\infty$ if and only if $\int_{\{\vert x\vert>1\}} g(x) \nu(dx)<\infty$. The claim follows.
\end{proof}

\begin{prop}
 Let $Y$ be an $H$-valued stochastic process with independent increments such that
\begin{itemize}
 \item $\<u\vert Y\>$ is a L\'evy process for every $u\in H$.
\end{itemize}
Then $Y$ is a L\'evy process in law.
\end{prop}
\begin{rem}
 The authors do not know if property (1) in the assumption above is obsolete.
\end{rem}
\begin{proof}
 Define $f:\mathbb R_+\times H\rightarrow\mathbb C,(t,u)\mapsto \E\exp(i\<u\vert Y(t)\>)$. Then $f(t,\cdot)$ is the characteristic function of $Y(t)$ for any $t\geq0$. Moreover, $f(t,u)=\exp(t\rho(u))$ for some function $\psi:H\rightarrow\mathbb C$ because $\<u\vert Y(t)\>$ is a L\'evy process and $\rho(u)$ is its L\'evy exponent evaluated at $1$. Since $f(t,\cdot)$ is continuous for any $t>0$ we conclude that $\rho$ is continuous. Thus $f$ is contiuous. Consequently, $Y$ is stochastically continuous. Let $t,h\in\mathbb R_+$. Then
$$\E\exp(i\<u\vert Y(t+h)-Y(h)\>) = \E\exp(i(\<u\vert Y\>(t+h)-\<u\vert Y\>(h))) = f(t,u)$$
for any $t\in\mathbb R_+$, $u\in H$ because $\<u\vert Y\>$ is a L\'evy process. Thus $Y$ has stationary increments.
\end{proof}

\subsection{Estimates}
\begin{lem}\label{L:schranke der Wachstumsfunktion Poissonfall}
 Let $N_1,\dots,N_d$ be compound Poisson processes on $H_1,\dots,H_d$ and $g:H\rightarrow\mathbb R_+$ be measurable and subadditive with $g(0)=0$. Let $N:=(N_1,\dots,N_d)$ be a L\'evy process. Then
$$Eg(N(\theta)) \leq \vert\theta\vert\sqrt{d}\int_H g(x)\mu(dx)$$
for any $\theta\in\mathbb R_+^d$ where $\mu$ is the L\'evy measure (or jump intensity measure) of $N$.
\end{lem}
\begin{proof}
 Let $j\in\{1,\dots,d\}$, $\theta\in\mathbb R_+^d$ and define $g_j(x):=g(\eta_j(x))$ for any $x\in H_j$. \cite[Definition 4.14]{peszat.zabczyk.07} states that $P^{N_j(\theta_j)} = e^{-\lambda_j \theta_j}\exp(\theta\mu_j)$ where $\lambda_j:=\mu_j(H_j)$ and $\mu_j$ is the L\'evy measure of $N_j$. Moreover,
\begin{eqnarray*}
 \int_{H_j}g_j(x)(\mu_j^{*k})(dx) &\leq & \int_{H_j}\cdots\int_{H_j} (g_j(x_1)+\dots+g_j(x_k)) \mu_j(dx_1)\dots\mu_j(dx_k) \\
 &=& \int_{H_j}g_j(x)\mu_j(dx) k\lambda_j^{k-1}.
\end{eqnarray*}
for any $k\in\mathbb N$. Thus 
$$\int_{H_j}g_j(x)(\exp(\theta_j\mu_j))(dx) \leq \theta_j\int_Hg_j(x)\mu_j(dx)e^{\lambda_j \theta_j}$$
and hence 
$$Eg_j(N_j(\theta_j)) = \int_{H_j}g_j(x)P^{L_j(\theta_j)}\leq \theta_j\int_{H_j}g_j(x)\mu_j(dx).$$
Since $\int_{H_j}g_j(x)\mu_j(dx)\leq\int_Hg(x)\mu(dx)$ the assertet inequality follows.
\end{proof}
\begin{rem}
The inequality in the Lemma above is sharp. Indeed, if $N_1=\dots=N_d$ are the same Poisson process with intensity $1$, $g:\mathbb R^d\rightarrow \mathbb R,x\mapsto \sum_{j=1}^d\vert x_j\vert$ and $\theta=(1,\dots,1)$, then
$$\E g(N(\theta)) = d = \sqrt{d}\vert \theta\vert\int_H g(x)\mu(dx).$$
\end{rem}

\begin{lem}\label{L:schranke der Wachstumsfunktion Martingalfall}
 Let $M_1,\dots,M_d$ be a mean zero and square integrable L\'evy processes on $H_1,\dots,H_d$ respectively such that $M:=(M_1,\dots,M_d)$ is a L\'evy process. Let $\alpha\in(0,2]$. Then there is a constant $C>0$ such that
$$\E(\vert M(\theta)\vert^\alpha) \leq \vert \theta\vert^{\alpha/2}C$$ 
for any $\theta\in\mathbb R_+^d$. Moreover, the constant $C$ can be chosen as
$$\left(\tr(\Gamma)+\int\vert x\vert^2\mu_j(dx)\right)^{\alpha/2}$$
where $(\beta,\Gamma,\mu)$ is the characteristics of $M$.
\end{lem}
\begin{proof}
 The case $\alpha=2$ follows from the Pythagorean theorem. Indeed,
\begin{eqnarray*}
\E(\vert M(\theta)\vert^2) &=& \sum_{j=1}^d \E(\vert M_j(\theta_j)\vert^2) \\
 &=& \sum_{j=1}^d \theta_j\E(\vert M_j(1)\vert^2) \\
 &\leq& \vert \theta\vert \E(\vert M(1)\vert^2) \\
 &=& \vert \theta\vert C
\end{eqnarray*}
for any $\theta\in\mathbb R_+^d$ where $C:=\E(\vert M(1)\vert^2)$. 

The other cases follow from the case $\alpha=2$ by Jensens inqueality. Indeed, we have
\begin{eqnarray*}
\E(\vert M(\theta)\vert^{\alpha}) &=& (\E(\vert M(\theta)\vert^2))^{\alpha/2} \\
  &\leq & (\vert \theta\vert C)^{\alpha/2} \\
  &=& \vert \theta\vert^{\alpha/2}C^{\alpha/2}.
\end{eqnarray*}
for any $\alpha\in(0,2]$ and any $\theta\in\mathbb R_+^d$.
\end{proof}

\begin{lem}\label{L:schranke der Wachstumsfunktion}
 Let $L$ be integrable. Then there are constants $C_1,C_2\in\mathbb R_+$ such that
$$\E\left(\vert L(\theta)\vert\right) \leq \vert \theta\vert C_1 + \vert \theta\vert^{1/2} C_2$$
for any $\theta\in\mathbb R_+^d$.
\end{lem}
\begin{proof}
 Define $g:=\vert\cdot\vert$. Then $g$ is subadditive. \cite[Theorem 4.23]{peszat.zabczyk.07} implies that $L(\theta)=a\theta+M(\theta)+N(\theta)$ for some $a\in H$ a mean zero and square integrable L\'evy process $M$ and a compound Poisson process $N$ where the L\'evy measure of $N$ is given by $\mu(A)=\nu(A\cap\{x\in H:\vert x\vert>1\}$ for any $A\in\mathcal B(H)$. The two previous Lemmas yield
\begin{eqnarray*}
 \E g(L(\theta)) &\leq& g(a\theta) + \E g(N(\theta)) + \E g(M(\theta)) \\
  &\leq & \vert \theta\vert\vert a\vert + \vert \theta\vert \int_{\vert x\vert>1}g(x)\mu(dx) + \vert \theta\vert^{1/2} C_2
\end{eqnarray*}
for some constant $C_2>0$ and any $\theta\in\mathbb R_+^d$. \cite[Proposition 4.18]{peszat.zabczyk.07} yields that $C_3:=\int_{\vert x\vert>1}g(x)\mu(dx)<\infty$. Define $C_1:=\vert a\vert+C_3$. Then
 $$\E(\vert L(\theta)\vert) \leq \vert \theta\vert\vert C_1 + \vert\theta\vert^{1/2} C_2$$
as claimed.
\end{proof}

We now state some technical Lemmas which are needed for the proof of Theorem \ref{S:subordiniertes tripel}. The first one essentially states that $\mathbb R_+^d\rightarrow \mathbb R,\theta\mapsto E f(L(\theta))$ growth at most linearly for smooth functions $f$. The second one states that this is also true for $\theta\mapsto \E\chi(L(\theta))$.
\begin{lem}\label{L:Wachstum}
 Let $f:H\rightarrow \mathbb R$ be bounded and uniformly continuous such that its derivatives up to order two are also bounded and uniformly continous. Then there is a constant $C>0$ such that
$$\vert \E(f(L(\theta)))-f(0)\vert \leq \sum_{j=1}^d\theta_jC$$
for any $\theta\in\mathbb R_+^d$. Moreover, the constant $C$ can be chosen as
$$\sup_{x\in H}2\vert f(x)\vert\nu(\vert x\vert>1)+\vert b\vert\Vert Df(x)\Vert_{H'}+\frac{1}{2}\left(\tr(Q)+\int_{\{\vert x\vert\leq 1}\vert x\vert^2\nu(dx)\right)\Vert D^2f(x)\Vert_{\mathrm{op}}$$
where $\Vert\cdot\Vert_{H'}$ denotes the operator norm on the space of linear functionals from $H$ to $\mathbb R$ and $\Vert\cdot\Vert_{\mathrm{op}}$ denotes the operator norm on the space of linear functions on $H$.
\end{lem}
\begin{proof}We first show the inequality for $d=1$. Let $UC_b^2$ be the set of functions which are bounded and uniformly continuous and whose derivatives up to order two are also bounded and uniformly continous. \cite[Theorem 5.4]{peszat.zabczyk.07} yields that $L$ is a Markov process and the domain of its generator $\mathcal A$ contains $UC_b^2$ and
$$\mathcal A f(x) = \<b,Df(x)\> +\frac{1}{2}\tr(Q)D^2f(x)+\int_H(f(x+y)-f(x)-1_{\{\vert y\vert\leq1\}}\<y,Df(x)\>)\nu(dy)$$
for any $x\in H$. In particular,
$$M(t) := f(L(t)) - \int_0^t \mathcal A f(L(s)) ds$$
is a martingale w.r.t.\ the filtration generated by $L$. Define $C:=\sup_{\{x\in H\}} \mathcal Af(x)$. The inequality follows.

For arbitrary $d$ the inequality follows from a simple induction.
\end{proof}

\begin{lem}\label{L:Wachstum chi}
 There is a constant $C>0$ such that
$$\vert \E\chi(L(\theta))\vert \leq \vert \theta\vert C$$
for any $\theta\in\mathbb R_+^d$.
\end{lem}
\begin{proof}
 Let $\chi_1: H\rightarrow H$ such that $\chi_1$ is twice continuously differentiable, its support is contained in the centered ball of radius $2$, $\chi_1(x)=x$ for $x\in H$ with $\vert x\vert\leq 1$ and its first two derivatives vanish on its zeros except for the zero in $0$. Then the restriction $f$ of $\vert\cdot\vert\circ\chi_1$ to the set $\{x\in H:\vert x\vert\geq1\}$ has a twice continuously differentiable continuation $\chi_2:H\rightarrow\mathbb R_+$ such that $\chi_2$ vanishes on the centered ball of radius $0.5$. By Lemma \ref{L:Wachstum} for each $\alpha\in H'$ there is a constant $C_\alpha$ such that 
$$\vert \E(\alpha\circ\chi_1)(\theta)\vert \leq C_\alpha\sum_{j=1}^d\theta_j\leq C_1\sqrt{d}\vert \theta\vert$$
for any $\theta\in\mathbb R_+^d$ where $C_1$ is defined by
\begin{eqnarray*}
&&\sup_{x\in H}\Bigg(2\vert \chi_1(x)\vert\nu(\vert x\vert>1)+\vert b\vert\Vert D\chi_1(x)\Vert_{op}\\
&&\hspace{1cm} +\frac{1}{2}\left(\tr(Q)+\int_{\{\vert x\vert\leq 1\}}\vert x\vert^2\nu(dx)\right)\Vert D^2\chi_1(x)(y,z)\Vert_{B^2} \Bigg)
\end{eqnarray*}
and where $\Vert\cdot\Vert_{\mathrm{op}}$ denotes the operator norm on the space of linear functions on $H$ and $\Vert\cdot\Vert_{B^2}$ denotes the operator norm of bilinear functions from $H\times H$ to $H$. For $\theta\in\mathbb R_+^d$ the Hahn-Banach theorem \cite[Theorem 3.2]{rudin.87} yields that there is $\alpha_\theta\in H'$ such that $\Vert \E\chi_1(L(\theta))\Vert = \alpha_\theta(E\chi_1(L(\theta)))$ and $\Vert\alpha_\theta\Vert_{op}=1$. Hence
$$\Vert \E\chi_1(L(\theta))\Vert \leq C_1\sqrt{d}\vert \theta\vert.$$
Lemma \ref{L:Wachstum} yields that there is a constant $C_2$ such that $\E\chi_2(L(\theta))\leq \vert \theta\vert C_2$. Then
\begin{eqnarray*}
 \vert \E\chi(L(\theta))\vert &\leq& \vert \E\chi_1(L(\theta))\vert + \vert \E(\chi_1-\chi)(L(\theta))\vert \\
 &\leq& C_1\sqrt{d}\vert \theta\vert + \E\chi_2(L(\theta)) \\
 &\leq& C\vert \theta\vert
\end{eqnarray*}
where $C:=\sqrt{d}C_1+C_2$.
\end{proof}

\begin{lem}\label{L:Wachstum g}
 Let $u\in H$ and 
$$g:\mathbb R_+^d\rightarrow\mathbb C,u\mapsto \E\left\vert e^{i\<u\vert L(\theta)\>}-1-i\<u\vert\chi(L(\theta))\>\right\vert.$$
Then there is a constant $C>0$ such that $\vert g(\theta)\vert\leq \vert\theta\vert C$ for any $\theta\in\mathbb R_+^d$.
\end{lem}
\begin{proof}
 Let $\chi_1: H\rightarrow H$ such that $\chi_1$ is twice continuously differentiable, its support is contained in the centered ball of radius $2$, $\chi_1(x)=x$ for $x\in H$ with $\vert x\vert\leq 1$ and its first two derivatives vanish on its zeros except for the zero in $0$. In the proof of Lemma \ref{L:Wachstum chi} we have shown that there is a constant $C_2$ such that
$$ \E\vert \chi_1(L(\theta))-\chi(L(\theta))\vert \leq \vert \theta\vert C_2.$$
Let
$$f:H\rightarrow\mathbb C,x\mapsto \vert e^{i\<u\vert x\>}-1-i\<u\vert\chi_1(x)\>\vert.$$
Lemma \ref{L:Wachstum} yields $\E f(L(\theta))\leq \vert\theta\vert C_1$ for some constant $C_1>0$. Thus
$$g(\theta) \leq \E f(L(\theta))+\vert u\vert \E\vert \chi_1(L(\theta))-\chi(L(\theta))\vert \leq (C_1+\vert u\vert C_2)\vert\theta\vert$$
for any $\theta\in\mathbb R_+^d$.
\end{proof}

\bibliographystyle{amsplain}
\bibliography{bib}
\end{document}